\documentclass[12pt]{amsart}
\usepackage{a4wide}
\usepackage{amssymb}
\usepackage{amsthm}
\usepackage{amsmath}
\usepackage{amscd}
\usepackage{verbatim}
\usepackage{color}
\usepackage[mathscr]{eucal}

\numberwithin{equation}{section}

\newtheorem{theorem}{Theorem}
\newtheorem{lemma}[theorem]{Lemma}
\newtheorem{corollary}[theorem]{Corollary}
\newtheorem{conjecture}[theorem]{\bf Conjecture}

\theoremstyle{remark}

\numberwithin{theorem}{section} \numberwithin{equation}{section}

\newcommand{\R}{\mathbb{R}}
\renewcommand{\H}{\mathbb{H}}
\newcommand{\C}{\mathbb{C}}

\newcommand{\Q}{\mathbb{Q}}

\newcommand{\Z}{\mathbb{Z}}
\newcommand{\N}{\mathbb{N}}
\newcommand{\SL}{{\text {\rm SL}}}

\DeclareMathOperator{\im}{Im}

\newcommand{\sgn}{\operatorname{sgn}}

\title{Sums of class numbers and mixed mock modular forms}
\author{Kathrin Bringmann}
\address{Mathematical Institute\\University of Cologne\\ Weyertal 86-90 \\ 50931 Cologne \\Germany}
\email{kbringma@math.uni-koeln.de}
\author{Ben Kane}
\address{Department of Mathematics, University of Hong Kong, Pokfulam, Hong Kong}
\email{bkane@maths.hku.hk}
\date{\today}
\thanks{The research of the first author was supported by the Alfried Krupp Prize for Young University Teachers of the Krupp Foundation.  This research was completed while the second author was a postdoc at the University of Cologne}
\subjclass[2010] {11E41, 11F37, 11F30}
\keywords{class numbers, Lerch sums, mixed mock modular forms, mock theta functions, modular forms}
\begin{document}
\begin{abstract}
In this paper, we consider sums of class numbers of the type\\
 $\sum_{m\equiv a\pmod{p}} H\left(4n-m^2\right)$, where $p$ is an odd prime, $n\in \N,$ and $a\in \Z$.  By showing that these are coefficients of mixed mock modular forms, we obtain explicit formulas.  Using these formulas for $p=5$ and $7$, we then prove a conjecture of Brown et al. in the case that $n=\ell$ is prime.
\end{abstract}

\maketitle

\section{Introduction and statements of results}

Let $H(n)$ denote the $n$th \begin{it}Hurwitz class number\end{it}, i.e., the number of equivalence classes of positive definite binary quadratic forms of discriminant $-n$ with the class containing $x^2+y^2$ weighted by $\frac{1}{2}$ and the class containing $x^2+xy+y^2$ weighted by $\frac{1}{3}$.  Moreover, by convention $H(0)=-\frac{1}{12}$.  Certain congruence classes appear as coefficients of weight $\frac{3}{2}$ Eisenstein series.  However, the generating function for all Hurwitz class numbers 
\begin{equation}\label{Hgen}
\mathcal{H}(q):=\sum_{n\geq 0} H(n)q^n
\end{equation}
is not itself modular, but rather mock modular \cite{HZ}.  Roughly speaking, this means that $\mathcal{H}$ may be naturally ``completed'' to a non-holomorphic modular form (a further description is given in Section \ref{sec:Hurwitz}).  Mock modular forms have since shown up in a variety of applications.  To name a few examples, Ramanujan's mock theta functions have been shown to be mock modular forms \cite{ZwegersThesis}, they have led to asymptotic and exact formulas in partition theory \cite{BringmannOnoInvent, BringmannOnoAnnals}, they are related to Lie superalgebras \cite{BringmannFolsom, BringmannOnoKac, Liesuper}, and they are connected to the quantum theory of black holes \cite{BM, DabMurZag}.

In this paper, we prove conjectures reminiscent of the famous identity (cf. p. 154 of \cite{Eichler})
\begin{equation}\label{eqn:1case}
\sum_{|m|<2\sqrt{\ell}} H\left(4\ell-m^2\right) = 2\ell,
\end{equation}
where $\ell$ is an odd prime.  More specifically, for a prime $p$, $a\in \Z$, and $n\in \N$, this paper is focused on sums of the type 
$$
H_{a,p}(n):=\sum_{\substack{|m|\leq 2\sqrt{n}\\ m\equiv a\pmod{p}}} H\left(4n-m^2\right).
$$

A number of identities for $H_{a,p}(n)$ were obtained in \cite{Hurwitz} for the special cases that $n=\ell$ is prime and $p=2$, $3$, $5$, or $7$.  To give an example indicative of the results in \cite{Hurwitz}, they proved in the special case that $p=5$ and $n=\ell$ is prime that
$$
H_{a,5}(\ell)= \begin{cases}
\frac{\ell-3}{2}&\text{if }a\equiv 0\pmod{5},\text{ and } \ell \equiv 4\pmod{5},\\
\frac{\ell-1}{2}&\text{if }a\equiv \pm 1\pmod{5},\text{ and } \ell \equiv 3\pmod{5},\\
\frac{\ell-1}{2}&\text{if }a\equiv \pm 2\pmod{5},\text{ and } \ell \equiv 2\pmod{5}.
\end{cases}
$$
In the cases $p=5$ and $p=7$, they were unable to completely classify $H_{a,p}(\ell)$, but conjecture a number of pleasant identities similar to \eqref{eqn:1case} based on computer data.
\begin{conjecture}\label{conj:5case}
For every $a,L\in \Z$, there exist constants $c_1,c_2\in \Q$ (given explicity in \eqref{eqn:5conjecture}) such that for every prime $\ell\equiv L\pmod{5}$, we have
$$
H_{a,5}(\ell) = c_{1} \ell + c_{2}.
$$
\end{conjecture}
When restricted to certain congruence classes  for $a$ and $\ell\pmod{7}$, they conjecture a similar formula for $H_{a,7}(\ell)$.
\begin{conjecture}\label{conj:7case}
If $L=3,5,6$ and $a\in \Z$ or $(a,L)\equiv (\pm 1,1)\pmod{7}$, then there exist constants $c_1,c_2\in \Q$ (given explicitly in \eqref{eqn:7conjecture}) such that for every prime $\ell\equiv L\pmod{7}$, we have
$$
H_{a,7}(\ell) = c_{1} \ell + c_{2}.
$$
\end{conjecture}
We settle these conjectures in this paper.
\begin{theorem}\label{thm:conj}
Conjectures \ref{conj:5case} and \ref{conj:7case} are true.
\end{theorem}
Theorem \ref{thm:conj} is implied by a more general theorem, which we next describe.  The key ingredient is to use the mock modularity of the Hurwitz class number generating function, whereas the authors of \cite{Hurwitz} only took advantage of the modularity of certain congruence classes.  In particular, $H_{a,p}(n)$ are the coefficients of what are known as \begin{it}mixed mock modular forms\end{it}, which are products of mock modular forms with modular forms.  

We construct explicit mixed mock modular forms with these coefficients from the product of $\mathcal{H}(q)$ with unary theta functions $\vartheta_{a,N}(\tau)$ defined in \eqref{eqn:unth} (where $N\in \N$ and $q:=e^{2\pi i \tau}$ throughout).  To completely describe these, for a series $f(\tau)=\sum_{n\in \Z} a(n) q^n$, we require the $d$th $U$-operator $f|U(d)(\tau):=\sum_{n\in \Z} a(nd) q^n$ and the twist of a series $f$ by a character $\chi$, i.e., $f\otimes \chi(\tau):=\sum_{n\in \Z} a(n)\chi(n) q^n$.  It is straightforward to see that for every $a\in\Z$ and prime $p$, one has
\begin{equation}\label{eqn:classsums}
\sum_{\substack{n\geq 0\\ p\nmid n}} H_{a,p}(n) q^n= \left(\mathcal{H}(q)\vartheta_{a,p}\left(\tau\right)\right)\Big|U(4)\otimes \chi_{p}^2.
\end{equation}
Our main theorem expresses the right-hand side of \eqref{eqn:classsums} in terms of generating functions for explicit divisor sums.  For this, define
$$
\mathcal{G}_{N,r}(q):=\sum_{n\geq 1}\sum_{\substack{dd'=n\\ d\equiv \pm r\pmod{N}\\ d'> d}} dq^n  + \sum_{n\geq 1} \left(Nn-r\right)q^{\left(Nn-r\right)^2}.
$$
Moreover, we let $S_{N,r}$ denote the operator $f | S_{N,r}(\tau) = \sum_{n\in \Z} a(Nn+r) q^{Nn+r}$ which sieves coefficients congruent to $r$ modulo $N$.  
\begin{theorem}\label{thm:gencomplete}
For every $a\in \Z$ and odd prime $p$, we have that
\begin{equation}\label{eqn:gencomplete2}
\left(\mathcal{H}(q)\vartheta_{a,p}\left(\tau\right)\right)\Big|U(4)\otimes \chi_{p}^2  + \sum_{\substack{b\pmod{p}\\ b\not\equiv \pm a\pmod{p}}} \mathcal{G}_{p,a+b}(q)\Big|S_{p,a^2-b^2}
\end{equation}
is a weight $2$ holomorphic modular form on $\Gamma_0\left(p^2\right)\cap \Gamma_1(p)$.  Moreover, if $a=0$, then \eqref{eqn:gencomplete2} is a weight 2 modular form on $\Gamma_0\left(p^2\right)$.
\end{theorem}
This paper is organized as follows.  In Section \ref{sec:Hurwitz}, we recall Hirzebruch and Zagier's completion of the class number generating function and introduce important series known as Appell-Lerch sums.  In Section \ref{sec:mixed}, we show how to complete the functions on either side of \eqref{eqn:gencomplete2} to obtain non-holomorphic modular forms.  We conclude the paper by proving Theorem \ref{thm:gencomplete} and Theorem \ref{thm:conj} in Section \ref{sec:main}.

\section{Hurwitz class numbers, Lerch sums, and known facts}\label{sec:Hurwitz}

As mentioned in the introduction, the generating function \eqref{Hgen} for the Hurwitz class numbers is a mock modular form.  More precisely, Hirzebruch and Zagier \cite{HZ} proved that one can complete $\mathcal{H}$ by adding a certain simple sum involving the incomplete Gamma function
$$
\Gamma\left(s;y\right):=\int_{y}^{\infty} t^{s-1}e^{-t}dt.
$$
The completed function is non-holomorphic, but belongs to a special class of functions known as harmonic weak Maass forms.  To define these, we require the \begin{it}weight $k$ hyperbolic Laplacian\end{it} (denoting $\tau=x+iy$ with $x,y\in \R$ throughout)
$$
\Delta_{k}:=-y^2\left(\frac{\partial^2}{\partial x^2} +\frac{\partial^2}{\partial y^2}\right) +iky\left(\frac{\partial}{\partial x} + i \frac{\partial}{\partial y}\right).
$$
Weight $k$ \begin{it}harmonic weak Maass forms\end{it} for $\Gamma\subset\SL_2(\Z)$ are real analytic functions $\mathcal{F}:\H\to\C$ satisfying the following properties:
\noindent

\noindent
\begin{enumerate}
\item $\mathcal{F}|_{k} \gamma\left(\tau\right) = \mathcal{F}\left(\tau\right)$ for every $\gamma\in \Gamma$,
\item $\Delta_{k}\left(\mathcal{F}\right)=0$,
\item $\mathcal{F}$ has at most linear exponential growth at each cusp of $\Gamma$.
\end{enumerate}
Here $|_k$ is the usual weight $k$ slash operator.  The class number generating function also belongs to a distinguished subspace, consisting of those forms whose $n$th coefficient (in the Fourier expansion in $x$) vanishes unless $(-1)^{k-\frac{1}{2}}n\equiv 0,1\pmod{4}$, known as \begin{it}Kohnen's plus space\end{it} \cite{Kohnen}.

We collect the modularity properties of $\mathcal{H}$ in the following theorem which can be easily concluded from Theorem 2 of \cite{HZ}.
\begin{theorem}\label{thm:Hcomplete}
The function 
\begin{equation}\label{Hshadow}
\widehat{\mathcal{H}}(\tau):=\mathcal{H}(q) + \frac{1}{4\sqrt{\pi}}\sum_{n>0}n\Gamma\left(-\frac12; 4\pi n^2 y\right)q^{-n^2}+\frac{1}{8\pi \sqrt{y}}
\end{equation}
is a  weight $\frac{3}{2}$ harmonic weak Maass form on $\Gamma_0(4)$ in Kohnen's plus space.
\end{theorem}

A number of other completions help us to prove Theorem \ref{thm:gencomplete}.  To this end, for $u\in \C\setminus \left(\Z\tau+\Z\right)$ we define the \begin{it}multivariable Appell function\end{it} \cite{Liesuper,ZwegersAppell}
\begin{equation}\label{eqn:A2def}
A_{\ell}\left(u,v;\tau\right):=e^{\pi i\ell u}\sum_{n\in \Z} \frac{(-1)^{\ell n}q^{\frac{\ell}{2} n\left(n+1\right)}e^{2\pi i n v}}{1-e^{2\pi i u} q^n}. 
\end{equation}
Directly from the definition, we have
\begin{align}
\nonumber
(-1)^{\ell}A_{\ell}(u+1,v;\tau) &= A_{\ell}(u,v+1;\tau) = A_{\ell}(u,v;\tau),\\
\label{eqn:Ashift}
A_{\ell}\left(u+\tau,v+\ell\tau;\tau\right)&=(-1)^{\ell}q^{-\frac{\ell}{2}}e^{2\pi i v}A_{\ell}(u,v;\tau).
\end{align}
In order to add a non-holomorphic function which ``completes'' $A_{\ell}$ to satisfy modularity, we define
\begin{align*}
\vartheta\left(z;\tau\right)&:=\sum_{n\in \frac{1}{2}+\Z}e^{\pi i n^2\tau+2\pi i n\left(z+\frac{1}{2}\right)},\\
R\left(u;\tau\right)&:=\sum_{n\in \frac{1}{2}+\Z} \left(\sgn\left(n\right) - E\left(\left(n+\frac{\im(u)}{y}\right)\sqrt{2y}\right)\right)\left(-1\right)^{n-\frac{1}{2}}e^{-\pi i n^2\tau-2\pi i n u },\\
E\left(z\right)&:=2\int_0^z e^{-\pi t^2}dt.
\end{align*}
We note that for $u\neq 0$, we also have the useful formula
\begin{equation}\label{eqn:Egamma}
E(u)=\sgn(u)\left(1-\frac{e^{-\pi u^2}}{\pi |u|}+\frac{1}{2\sqrt{\pi}}\Gamma\left(-\frac12; \pi u^2\right)\right).
\end{equation}

Theorem 2.2 of \cite{ZwegersAppell} yields the transformation properties of the completion of $A_{\ell}$:
\begin{multline}\label{eqn:A2complete}
\widehat{A}_{\ell}\left(u,v;\tau\right):=A_{\ell}(u, v; \tau)\\
+\frac{i}{2}\sum_{k=0}^{\ell-1} e^{2\pi iku}\vartheta\left(v+k\tau+\frac{\ell-1}{2}; \ell\tau\right)R\left(\ell u-v-k\tau-\frac{\ell-1}{2}; \ell\tau\right).
\end{multline}
\begin{theorem}
The function $\widehat{A}_{\ell}$ satisfies
\begin{equation}\label{eqn:A2modularity}
\widehat{A}_{\ell}\left(\frac{u}{c\tau+d},\frac{v}{c\tau+d};\frac{a\tau+b}{c\tau+d}\right)=\left(c\tau+d\right) e^{\frac{\pi i c}{c\tau+d}\left(-\ell u^2+2uv\right)}\widehat{A}_\ell\left(u,v; \tau\right).
\end{equation}
Moreover, for every $m_1,m_2,n_1,n_2\in \Z$, we have 
\begin{multline}\label{eqn:A2elliptic}
\widehat{A}_{\ell}\left(u+n_1\tau+m_1,v+n_2\tau+m_2;\tau\right)\\ 
= (-1)^{\ell\left(n_1+m_1\right)}e^{2\pi i u\left(\ell n_1 -n_2\right)}e^{-2\pi i n_1v}q^{\frac{\ell}{2}n_1^2-n_1n_2}\widehat{A}_{\ell}\left(u,v;\tau\right).
\end{multline}
\end{theorem}

\section{Mixed mock modular forms and non-holomorphic completions}\label{sec:mixed}
In this section, we show how to add a non-holomorphic function to ``complete'' each of the functions defined in the introduction to obtain a non-holomorphic modular form.  In order to write down these completions, for $N\in \N$ it is helpful to define the function 
$$
R_{a,N}(\tau):=\sum_{\substack{n\in \Z\\ n \equiv a\pmod{N}}} |n|\Gamma\left(-\frac{1}{2};4\pi n^2 y\right) q^{-n^2}.
$$
Note that $R_{a,N}$ only depends on $a\pmod{N}$ and 
$$
R_{-a,N}=R_{a,N}.
$$
We further define the \begin{it}unary theta functions\end{it}
\begin{equation}\label{eqn:unth}
\vartheta_{a,N}\left(\tau\right):=\sum_{\substack{m\in \Z\\ m\equiv a\pmod{N}}}q^{m^2}.
\end{equation}

\subsection{}
In this subsection, we complete $\left(\mathcal{H}(q)\vartheta_{2a,p}(\tau)\right)\big|U(4)\otimes\chi_p^2$.
\begin{lemma}\label{lem:Htwist}
For every $a,b\in \Z$ and odd prime $p$, the function 
\begin{multline}\label{eqn:Htwist}
\left(\mathcal{H}(q)\vartheta_{2a,p}(\tau)\right)\Big|U(4)\otimes\chi_p^2+\frac{1}{8\sqrt{\pi}}\sum_{\substack{b\pmod{p}\\ b\not\equiv \pm a \pmod{p}}}\sum_{k=0}^1 R_{2b+kp,2p}\left(\frac{\tau}{4}\right)\vartheta_{2a+kp,2p}\left(\frac{\tau}{4}\right)\\
 + \frac{\vartheta_{2a,2p}\left(\frac{\tau}{4}\right)}{4\pi\sqrt{y}}
\end{multline}
satisfies weight $2$ modularity for $\Gamma:=\Gamma_0\left(p^2\right)\cap \Gamma_1(p)$.
\end{lemma}
\begin{proof}
By Proposition 2.1 of \cite{Shimura}, $\vartheta_{2a,p}$ is modular of weight $\frac{1}{2}$ on $\Gamma_0\left(4p^2\right)\cap \Gamma_1(p)$.

We first write 
\begin{equation}\label{eqn:HThrewrite}
\left(\mathcal{H}(q)\vartheta_{2a,p}(\tau)\right)\big|U(4)\otimes\chi_p^2= \sum_{1\leq r\leq p-1} \left(\mathcal{H}(q)\vartheta_{2a,p}(\tau)\right)\big|U(4)\big|S_{p,r}.
\end{equation}
By Lemma 1 of \cite{Li}, $(\widehat{\mathcal{H}}(q)\vartheta_{2a,p}(\tau))|U(2)$ fulfills weight 2 modularity on $\Gamma_0(2p^2)\cap \Gamma_1(2p)$.  Since the $n$th coefficient of $(\widehat{\mathcal{H}}(q)\vartheta_{2a,p}(\tau))|U(2)$ is zero unless $n$ is even, Lemma 4 of \cite{Li} implies that $(\widehat{\mathcal{H}}(q)\vartheta_{2a,p}(\tau))|U(4)$ satisfies weight 2 modularity on $\Gamma_0(p^2)\cap \Gamma_1(p)$.  Moreover, since modular forms on $\Gamma_1(p)$ split into modular forms on $\Gamma_0(p)$ with Nebentypus, rewriting $\chi_p^2 = 1-U(p)V(p)$ (where as usual $f|V(d)(\tau):=f(d\tau)$), Lemma 1 of \cite{Li} implies that the level goes down with $U(p)$ and back up with $V(p)$, so that overall the group becomes $\Gamma$.

By Theorem \ref{thm:Hcomplete}, to complete each summand $\left(\mathcal{H}(q)\vartheta_{2a,p}(\tau)\right)\big|U(4)\big|S_{p,r}$ on the right-hand side of \eqref{eqn:HThrewrite}, one must add
\begin{equation}\label{eqn:Sievecomplete}
\left(\frac{1}{4\sqrt{\pi}}\sum\limits_{\substack{n>0\\ m\equiv 2a\pmod{p}}} n\Gamma \left(-\frac12; 4\pi n^2 y\right) q^{m^2-n^2}+\frac{1}{8\pi\sqrt{y}}\sum_{m\equiv 2a\pmod{p}}q^{m^2}\right)\Bigg|U(4)\Bigg|S_{p,r}.
\end{equation}
Due to $m\equiv 2a\pmod{p}$ and the congruences implied by $U(4)$ and $S_{p,r}$, the congruence conditions on $n$ and $m$ are equivalent to $m\equiv n\pmod{2}$, $m\equiv 2a\pmod{p}$, and $n\equiv \pm 2b \pmod{p}$, where $b$ satisfies $a^2-b^2\equiv r\pmod{p}$.  If no such $b$ exists, then \eqref{eqn:Sievecomplete} equals zero.  We may thus  assume that such a $b$ exists.  Whenever $b\not\equiv 0\pmod{p}$, the fact that $p$ is odd implies that \eqref{eqn:Sievecomplete} equals
$$
\frac{1}{8\sqrt{\pi}}\left(\sum_{k=0}^1\sum_{\pm} R_{\pm 2b+kp,2p}\left(\frac{\tau}{4}\right)\vartheta_{2a+kp,2p}\left(\frac{\tau}{4}\right)\right).
$$

Moreover, in the case that $b\equiv 0\pmod{p}$, \eqref{eqn:Sievecomplete} equals
$$
\frac{1}{8\sqrt{\pi}}\left(\sum_{k=0}^1 R_{kp,2p}\left(\frac{\tau}{4}\right)\vartheta_{2a+kp,2p}\left(\frac{\tau}{4}\right)\right)+\frac{\vartheta_{2a,2p}\left(\frac{\tau}{4}\right)}{4\pi\sqrt{y}}.
$$
To finish the proof, one then sums over all choices of $b$ to obtain \eqref{eqn:Htwist}.
\end{proof}

\subsection{}  In this section, we complete $\sum_{\pm}\mathcal{G}_{p,a\pm b}(q)|S_{p,a^2-b^2}$.

We begin with a lemma, which gives a more useful form ($\ell\in \Z$) for
$$
\mathcal{R}_{\ell,p}(v;\tau):=\left[\frac{d}{dv} \left(e^{\frac{\pi i\ell v}{p}}q^{-\frac{\ell^2}{4}}R\left(p\ell\tau-v-\frac12; 2p^2\tau\right)\right)\right]_{v=0}.
$$
\begin{lemma}\label{lem:diffR}
If $p$ is an odd prime and $\ell\in\Z$ satisfies $-p< \ell\leq p$, then 
$$
\mathcal{R}_{\ell,p}(v;\tau)=\sqrt{\pi}\sum_{n\in\frac12+\Z}\left\lvert n+\frac{\ell}{2p}\right\rvert \Gamma\left(-\frac12; 4\pi p^2\left(n+\frac{\ell}{2p}\right)^2 y\right) q^{-p^2\left(n+\frac{\ell}{2p}\right)^2}+\frac{\delta_{\ell=p}}{p\sqrt{y}}.
$$
\end{lemma}
\begin{proof}
We have
\begin{multline*}
e^{\frac{\pi i\ell v}{p}}q^{-\frac{\ell^2}{4}}R\left(p\ell\tau-v-\frac12; 2p^2\tau\right)\\
=i\sum_{n\in\frac12+\Z}\left(\sgn(n)-E\left(2p\left(n+\frac{\ell}{2p}-\frac{\text{Im}(v)}{2p^2y}\right)\sqrt{y}\right)\right)
q^{-p^2\left(n+\frac{\ell}{2p}\right)^2}e^{2\pi iv\left(n+\frac{\ell}{2p}\right)}.
\end{multline*}
We now note that whenever $n+\frac{\ell}{2p}\neq 0$, we have 
\[
\sgn(n)=\sgn\left(n+\frac{\ell}{2p}\right).
\]
Differentiating and plugging in \eqref{eqn:Egamma} yields the claim, after a straightforward calculation.
\end{proof}

In order to complete the right-hand side of \eqref{eqn:gencomplete2}, we pair the terms $b$ and $-b$ and determine the associated completion.
\begin{lemma}\label{lem:Gcompletegen}
Suppose that $a,b\in \Z$ and $p$ is an odd prime with $a\not\equiv \pm b\pmod{p}$. 

\noindent
\begin{enumerate}
\item
If $b\not\equiv 0\pmod{p}$, then the function 
$$
\widehat{\mathcal{G}}_{p,a,b}(q):=\sum_{\pm}\mathcal{G}_{p,a\pm b}(q)\Big|S_{p,a^2-b^2}-\frac{1}{4\sqrt{\pi}}\sum_{k=0}^1 R_{kp+2b,2p}\left(\frac{\tau}{4}\right)\vartheta_{kp+2a,2p}\left(\frac{\tau}{4}\right)
$$
satisfies weight $2$ modularity for $\Gamma$.
\item
If $b\equiv 0\pmod{p}$, then the function 
$$
\widehat{\mathcal{G}}_{p,a,0}(q):=\mathcal{G}_{p,a}(q)\Big|S_{p,a^2}-\frac{1}{8\sqrt{\pi}}\sum_{k=0}^1 R_{kp,2p}\left(\frac{\tau}{4}\right)\vartheta_{kp+2a,2p}\left(\frac{\tau}{4}\right)  - \frac{\vartheta_{2a,2p}\left(\frac{\tau}{4}\right)}{4\pi\sqrt{y}}
$$
satisfies weight $2$ modularity for $\Gamma$.
\item
In the case that $a\equiv 0\pmod{p}$, the function
$$
\sum_{0<b<\frac{p}{2}}\widehat{\mathcal{G}}_{p,0,b}\Big|S_{p,-b^2}
$$
furthermore satisfies weight 2 modularity for $\Gamma_0\left(p^2\right)$.
\end{enumerate}
\end{lemma}
\begin{proof}
We first assume that $b\not\equiv 0\pmod{p}$ and may assume without loss of generality that $a+b<p$, since the definition of $\mathcal{G}_{p,a+b}$ only depends on $a$ and $b$ modulo $p$.  Recalling the definition \eqref{eqn:A2def} of $A_2$, we obtain
\begin{align*}
\mathcal{G}_{p, a+b}(q)\Big| S_{p,a^2-b^2}=&\sum_{n\equiv a^2-b^2 \pmod{p}}\sum\limits_{{dd'=n\atop{d\equiv \pm (a+b)\pmod{p}}}\atop{d'> d}}d q^n\\
=&\sum\limits_{n\geq 0\atop{m\geq 0}}\left(pn+(a+b)\right)q^{\left(pn+(a+b)\right)\left(pn+p+(a-b)+pm\right)}\\
&+\sum_{\substack{n\geq 1\\ m\geq 0}}\left(pn-(a+b)\right)q^{\left(pn-(a+b)\right)\left(pn+(b-a)+pm\right)}\\
=&\sum_{n\in\Z}\frac{\left(pn+(a+b)\right)q^{\left(pn+(a+b)\right)\left(pn+p+(a-b)\right)}}{1-q^{p\left(pn+(a+b)\right)}}\\
=&\frac{p}{2\pi i}\left[\frac{d}{dv}\left(A_2\left(\left(a+b\right)p\tau, v+2ap\tau; p^2\tau\right) e^{\frac{2\pi i\left(a+b\right)v}{p}}q^{a^2-b^2}\right)\right]_{v=0}.
\end{align*}
The analogous calculation for $-b$ follows by \eqref{eqn:Ashift}.  A similar calculation yields a uniform equation for $b\equiv 0\pmod{p}$ and $a\not\equiv 0\pmod{p}$ as well.  Overall, we obtain
$$
\mathcal{G}_{p, a\pm b}(q)\Big| S_{p,a^2-b^2}=\frac{p}{2\pi i}\left[\frac{d}{dv}\left(A_2\left(\left(b\pm a\right)p\tau, v\pm 2ap\tau; p^2\tau\right) e^{\frac{2\pi i\left(b\pm a\right)v}{p}}q^{a^2-b^2}\right)\right]_{v=0}.
$$

We next prove the modularity of 
\begin{equation}\label{eqn:hatG}
\widehat{G}_b(\tau):=\frac{p}{2\pi i}\left[\frac{d}{dv}\left(\widehat{A}_2\left(\left(a+b\right)p\tau, v+2ap\tau; p^2\tau\right) e^{\frac{2\pi i\left(a+b\right)v}{p}}q^{a^2-b^2}\right)\right]_{v=0}.
\end{equation}
For $\left(\begin{smallmatrix} \alpha &\beta\\ \gamma&\delta\end{smallmatrix}\right)\in \Gamma$, we use equation \eqref{eqn:A2modularity} followed by \eqref{eqn:A2elliptic} with $m_1=(a+b)\beta p$, $m_2=a\beta p$, $n_1=\left(b+a\right)\left(\frac{\alpha-1}{p}\right)\in \Z$, and $n_2=2a\left(\frac{\alpha-1}{p}\right)\in\Z$ to yield
$$
\widehat{G}_b\left(\frac{\alpha\tau+\beta}{\gamma\tau+\delta}\right) = \left(\gamma\tau+\delta\right)^2\widehat{G}_b(\tau).
$$
Hence $\widehat{G}_b$ satisfies weight 2 modularity.

It remains to compute the non-holomorphic part of $\sum_{\pm}\widehat{G}_{\pm b}$.  By \eqref{eqn:A2complete} and the definition \eqref{eqn:hatG} of $\widehat{G}_b$, we have  
\begin{multline}\label{eqn:diffpm}
\sum_{\pm} \widehat{G}_{\pm b}(\tau) = \sum_{\pm} \mathcal{G}_{p, a\pm b}(q)\Big|S_{p,a^2-b^2} + \frac{p}{4\pi}\Bigg[\frac{d}{dv}\sum_{\pm}\sum_{k=0}^{1} \vartheta\left(v\pm 2ap\tau+kp^2\tau +\frac{1}{2};2p^2\tau\right)\\
\times R\left(2bp\tau -v-kp^2\tau -\frac{1}{2};2p^2\tau\right)e^{2\pi i\left(a^2-b^2+ k\left(b\pm a \right)p\right)\tau} e^{\frac{2\pi i (b\pm a)v}{p}} \Bigg]_{v=0}.
\end{multline}
Here we double count $b\equiv 0\pmod{p}$ to get a uniform formula.  We now rewrite
\begin{multline*}
q^{\frac{1}{4}\left(kp\pm 2a\right)^2}e^{\frac{2\pi i v}{p}\left(\frac{kp}{2} \pm 2a\right)} \vartheta\left(v\pm 2ap\tau+kp^2\tau +\frac{1}{2};2p^2\tau\right)\\
=-\sum_{m\in \frac{1}{2}+\Z} q^{\left(mp+\frac{kp}{2}\pm a\right)^2} e^{\frac{2\pi iv}{p}\left(mp+\frac{kp}{2}\pm a\right)}.
\end{multline*}
However, we have that
\begin{multline*}
\sum_{\pm}\left[\frac{d}{dv}\sum_{m\in \frac{1}{2}+\Z} q^{\left(mp+\frac{kp}{2}\pm a\right)^2} e^{\frac{2\pi iv}{p}\left(mp+\frac{kp}{2}\pm a\right)}\right]_{v=0}\\
=\frac{2\pi i}{p}\sum_{\pm}\sum_{m\in \frac{1}{2}+\Z}\left(mp+\frac{kp}{2}\pm a\right) q^{\left(mp+\frac{kp}{2}\pm a \right)^2}=0,
\end{multline*}
which can be seen by making the change of variables $m\to -m-k$.  Moreover, by taking $m\to -m-k$, we also see that
$$
\sum_{\pm}\sum_{m\in \frac{1}{2}+\Z} q^{\left(mp+\frac{kp}{2}\pm a \right)^2}=2\sum_{m\in \frac{1}{2}+\Z} q^{\frac{1}{4}\left(2mp+kp+ 2a \right)^2}=2\vartheta_{(1-k)p+2a,2p}\left(\frac{\tau}{4}\right).
$$
It follows that \eqref{eqn:diffpm} equals
$$
-\frac{p}{2\pi}\sum_{k=0}^1 \mathcal{R}_{2b-kp,p}\left(v;\tau\right)\vartheta_{(1-k)p+2a,2p}\left(\frac{\tau}{4}\right).
$$
By Lemma \ref{lem:diffR}, we may rewrite this as 
\begin{multline*}
-\delta_{b=p}\frac{\vartheta_{2a,2p}\left(\frac{\tau}{4}\right)}{2\pi\sqrt{y}}-\frac{1}{4\sqrt{\pi}}\sum_{k=0}^1 \vartheta_{(1-k)p+2a,2p}\left(\frac{\tau}{4}\right)\\
\times \sum_{n\in \frac{1}{2}+\Z} \left|2np-kp+2b\right|\Gamma\left(-\frac{1}{2};\pi\left(2np-kp+2b\right)^2y\right)q^{-\frac{1}{4}\left(2np-kp+2b\right)^2}\\
 =-\frac{1}{4\sqrt{\pi}}\sum_{k=0}^1 R_{(1-k)p+2b,2p}\left(\frac{\tau}{4}\right)\vartheta_{(1-k)p+2a,2p}\left(\frac{\tau}{4}\right)-\delta_{b=p}\frac{\vartheta_{2a,2p}\left(\frac{\tau}{4}\right)}{2\pi\sqrt{y}}.
\end{multline*}
Statements (1) and (2) now follow from Lemma \ref{lem:diffR} with $\ell=2b-kp$.

We next prove part (3).  The claim is equivalent to showing that 
$$
g(\tau):=\frac{4\pi i}{p} \sum_{0<b<\frac{p}{2}}\widehat{\mathcal{G}}_{p,0,b}\Big|S_{p,-b^2}=\left[\frac{d}{dv}\sum_{b\pmod{p}^\ast}e^{\frac{2\pi ib v}{p}}q^{-b^2}\widehat{A}_2\left(bp\tau, v; p^2\tau\right)\right]_{v=0}
$$
satisfies weight 2 modularity for $\Gamma_0\left(p^2\right)$.  Here the sum runs over those $b\pmod{p}$ with $(b,p)=1$.  Note that the sum only depends on $b\pmod{p}$ because, for $b'\equiv b\pmod{p}$, we may use \eqref{eqn:A2elliptic} with $n_1=\frac{b'-b}{p}$.  However, a simple calculation yields
\begin{multline*}
g\left(\frac{\alpha\tau+\beta}{\gamma\tau+\delta}\right)=(\gamma\tau+\delta)^2\left[\frac{d}{dv}\sum_{b\pmod{p}^\ast}q^{-\alpha^2b^2}e^{\frac{2\pi i b\alpha v}{p}}\widehat{A}_2\left(b\alpha p\tau, v; p^2\tau\right)\right]_{v=0}\\
=\left(\gamma\tau+\delta\right)^2 g(\tau),
\end{multline*}
where we used that $b\alpha$ runs $\pmod{p}^\ast$ if $b$ does.

\end{proof}


\section{Proof of Theorem \ref{thm:gencomplete} and Theorem \ref{thm:conj}}\label{sec:main}

In this section, we prove our main theorem and then give explicit identities for 
$$
\sum_{
\substack{n\geq 0\\ p\nmid n}} H_{a,p}(n)q^n
$$
for certain fixed choices of $a$ and $p$.  In particular, since the divisor sums occurring in Theorem \ref{thm:gencomplete} are particularly simple for primes, we obtain the desired conjectures leading to Theorem \ref{thm:conj}.  We begin with the proof of our main theorem.

\begin{proof}[Proof of Theorem \ref{thm:gencomplete}]
By fixing $a$ and summing over all congruence classes for $b\not\equiv \pm a \pmod{p}$ in Lemma \ref{lem:Gcompletegen},
we see that the non-holomorphic parts cancel the non-holomorphic part from Lemma \ref{lem:Htwist} and hence the sum is a (weakly) holomorphic modular form.  Every coefficient of our new overall function may be written as a linear combination of class numbers and divisor sums and hence grows polynomially. Thus we have a holomorphic modular form of weight $2$, yielding Theorem \ref{thm:gencomplete}.
\end{proof}
To compute explicit identities, we use the following lemma, which follows from the valence formula.
\begin{lemma}\label{lem:modcomplete}
If $p$ is an odd prime, $a\in \Z$, and $f$ is a holomorphic modular form of weight $2$ on $\Gamma$, then 
$$
\left(\mathcal{H}(q)\vartheta_{a,p}\left(\tau\right)\right)\Big|U(4)\otimes \chi_{p}^2  + \sum_{\substack{b\pmod{p}\\ b\not\equiv \pm a\pmod{p}}} \mathcal{G}_{p,a+b}(q)\Big|S_{p,a^2-b^2}=f(\tau)
$$
if and only if the first $\frac{p}{6}\left(p^2-1\right)$ Fourier coefficients agree.

Moreover, if $a\equiv 0\pmod{p}$ and $f$ satisfies weight $2$ modularity for $\Gamma_0(p^2)$, then the above identity holds if and only if it holds for the first $\frac{p}{6} \left(p+1\right)$ coefficients.
\end{lemma}

Denoting $\sigma(n):=\sum_{d\mid n} d$, we write the Eisenstein series part of the modular forms from Lemma \ref{lem:modcomplete} (in the special cases $p=3,5,7$) in terms of 
$$
\mathcal{D}(q):=\sum_{n=1}^{\infty} \sigma\left(n\right)q^n.
$$
Since $\mathcal{D}$ is essentially a constant multiple of the weight $2$ Eisenstein series $E_2$, it is well-known that
\begin{equation}\label{Dshadow}
\widehat{\mathcal{D}}(\tau):=\mathcal{D}(q)-\frac{1}{24}+\frac{1}{8\pi y}
\end{equation}
transforms like a modular form of weight $2$ on $\SL_2(\Z)$.  In particular, since every non-trivial character $\chi$ satisfies $\chi(0)=0$, the function $\mathcal{D}\otimes \chi$ is a weight 2 holomorphic modular form.  More precisely, if the modulus of $\chi$ is $m\in \N$, then $\mathcal{D}\otimes \chi$ is a weight 2 holomorphic modular form on $\Gamma_0(m^2)$ (cf. Proposition 2.8 of \cite{OnoBook}).  Furthermore, if $r\not\equiv 0\pmod{p}$, then a straightforward calculation shows that $\mathcal{D}\big|S_{p,r}$ is a holomorphic modular form of weight $2$ on $\Gamma$.  

It is well-known (cf. Section 7.2, Example 2 in \cite{DabMurZag}) that 
$$
\left(\mathcal{H}\vartheta_{0,1}\right)\Big|U(4)=2\mathcal{D}-\mathcal{G}_{1,0}-\frac{1}{12},
$$
while one sees directly that $\vartheta_{0,1}=\sum_{a\pmod{p}} \vartheta_{a,p}$ and $\vartheta_{a,p}=\vartheta_{-a,p}$.  Hence we only need to determine formulas for the modular forms from Lemma \ref{lem:modcomplete} whenever $0\leq a\leq \frac{p-3}{2}$ to obtain them for all $a\in \Z$.  After constructing such modular forms, we use \eqref{eqn:classsums} to conclude Theorem \ref{thm:conj}.  For simplicity, we only work out the exact identities for $a=0$.  

\subsection{$p=3$}
We include the case $p=3$ since the formulas are particularly simple in this case and because it indicates the general method well.  Using the fact that $\mathcal{D}\otimes \chi_3$ is a holomorphic modular form of weight $2$ for $\Gamma_0(9)$, Lemma \ref{lem:modcomplete} implies that
\begin{equation}\label{eqn:3}
\left(\mathcal{H}\vartheta_{0,3}\right)\Big|U(4)\otimes\chi_3^2=-2\mathcal{G}_{3, 1}\Big|S_{3, 2}+\mathcal{D}\otimes \chi_3^2-\frac14\mathcal{D}\otimes \chi_3\left(1+\chi_3\right).
\end{equation}
The $n$th coefficient of $-2\mathcal{G}_{3, 1}\Big|S_{3, 2}$ is $0$ unless $n\equiv 2\pmod{3}$, in which case it equals
\[
-2\sum\limits_{\substack{ d\equiv \pm 1\pmod{3}\\ d\mid n, d<\frac{n}{d}}}d.
\]
The $n$th coefficient of $\mathcal{D}\otimes \chi_3^2$ is $0$ if $3|n$ and otherwise $\sigma(n)$.  Finally, the $n$th coefficient of $-\frac14\mathcal{D}\otimes \chi_3(1+\chi_3)$ is $0$ unless $n\equiv 1\pmod{3}$, in which case it equals $-\frac12 \sigma(n)$.  Thus the overall $n$th coefficient on the right-hand side of \eqref{eqn:3} is
\[
\begin{cases}
\frac12 \sigma(n)&\quad\text{ if } n\equiv 1\pmod{3},\\
\sigma(n)-2\sum\limits_{d\equiv \pm 1\pmod{3}\atop{d\mid n, d<\frac{n}{d}}}d &\quad\text{ if } n\equiv 2\pmod{3}.
\end{cases}
\]
Comparing with \eqref{eqn:classsums}, we get in particular for a prime $\ell>3$
\[
H_{0,3}(\ell)= 
\begin{cases}
\frac{\ell+1}{2}&\quad\text{ if } \ell\equiv 1\pmod{3},\\
\ell-1 &\quad\text{ if } \ell\equiv 2\pmod{3}.
\end{cases}
\]

\subsection{$p=5$}

The precise version of Conjecture \ref{conj:5case} in \cite{Hurwitz} is given by the following.
\begin{conjecture}\label{conj:5}
For a prime $\ell$ and $a\in \Z$ one has that 
\begin{equation}\label{eqn:5conjecture}
H_{a,5}(\ell)=
\begin{cases}
\frac{\ell+1}{2}&\text{if }a\equiv 0\pmod{5},\text{ and } \ell \equiv 1\pmod{5},\\
\frac{\ell+1}{3}&\text{if }a\equiv 0\pmod{5},\text{ and } \ell \equiv 2,3\pmod{5},\\
\frac{\ell+1}{3}&\text{if }a\equiv \pm 1\pmod{5},\text{ and } \ell \equiv 1,2\pmod{5},\\
\frac{5\ell+5}{12}&\text{if }a\equiv \pm 1\pmod{5},\text{ and } \ell \equiv 4\pmod{5},\\
\frac{5\ell-7}{12}&\text{if }a\equiv \pm 2\pmod{5},\text{ and } \ell \equiv 1\pmod{5},\\
\frac{\ell+1}{3}&\text{if }a\equiv \pm 2\pmod{5},\text{ and } \ell \equiv 3,4\pmod{5}.
\end{cases}
\end{equation}
\end{conjecture}
We prove Conjecture \ref{conj:5} by showing the following more precise version.
\begin{corollary}\label{cor:5case}
One has that 
\begin{align*}
\left(\mathcal{H}\vartheta_{0,5}\right)\Big|U(4)\otimes\chi_5^2=&\frac{1}{2}\mathcal{D}\otimes\chi_5^2-\frac{1}{12}\mathcal{D}\otimes \chi_5\left(\chi_5-1\right) - 2 \mathcal{G}_{5,1}\Big|S_{5,4} - 2 \mathcal{G}_{5,2}\Big|S_{5,1},\\
\left(\mathcal{H}\vartheta_{1,5}\right)\Big|U(4)\otimes\chi_5^2=&\frac{1}{3}\mathcal{D}\otimes\chi_5^2+\left(\frac{1}{6}\mathcal{D}-\mathcal{G}_{5,1}-\mathcal{G}_{5,2}\right)\Big|S_{5,3}\\
&\qquad\qquad\qquad\qquad\quad+\left( \frac{1}{12}\mathcal{D}-\frac{1}{2}\mathcal{G}_{5,2}-\frac{1}{2}\mathcal{G}_{5,3}\right)\Big|S_{5,4}.
\end{align*}

In particular, the conjectured formula \eqref{eqn:5conjecture} is true.
\end{corollary}

\begin{proof}
The holomorphic modular forms occurring by twisting $\mathcal{D}$ with a character of modulus 5 have level $\Gamma_0(25)$, while $\mathcal{D}|S_{5,a}$ ($a\not\equiv 0\pmod{5}$) is a holomorphic modular form for $\Gamma$.  After checking $20$ coefficients, Lemma \ref{lem:modcomplete} yields the equalities claimed in the corollary.  To obtain \eqref{eqn:5conjecture}, we simply apply an analysis similar to that used in the case for $p=3$ above to obtain the explicit coefficients.
\end{proof}

\subsection{$p=7$}
Conjecture \ref{conj:7case} follows from the following more precise version.
\begin{conjecture}\label{conj:7}
For a prime $\ell$ and $a\in\Z$ one has that
\begin{equation}\label{eqn:7conjecture}
H_{a,7}(\ell)= 
\begin{cases}
\frac{\ell+1}{3}&\text{if }a\equiv \pm 1\pmod{7},\text{ and } \ell \equiv 1\pmod{7},\\
\frac{\ell+1}{4}&\text{if }a\equiv \pm 1\pmod{7},\text{ and } \ell \equiv 3,6\pmod{7},\\
\frac{\ell+1}{4}&\text{if }a\equiv \pm 2\pmod{7},\text{ and } \ell \equiv 3,5\pmod{7},\\
\frac{\ell+1}{4}&\text{if }a\equiv \pm 3\pmod{7},\text{ and } \ell \equiv 5,6\pmod{7},\\
\end{cases}
\end{equation}
\end{conjecture}

The argument is analogous to the above cases for $p=3$ and $p=5$, except that there are cusp forms in the spaces of interest.  For example, for $a=0$, we require the weight $2$ newform (of level $49$) $g_7$ associated to the elliptic curve $y^2+xy=x^3-x^2-2x-1$ (which is denoted $49A1$ and has CM by $\Q\left(\sqrt{-7}\right)$).  The coefficients of $g_7$ are multiplicative and Parry has explicitly written $g_7$ in terms of the two variable Ramanujan theta function.

After comparing the first $56$ coefficients in the identity, Lemma \ref{lem:modcomplete} immediately yields the following corollary.
\begin{corollary}\label{cor:7casea=0}
One has that 
\begin{multline*}
\left(\mathcal{H}\vartheta_{0,7}\right)\Big|U(4)\otimes\chi_7^2=\frac{1}{4}\mathcal{D}\otimes\chi_7^2+\frac{1}{24}\mathcal{D}\otimes \chi_7\left(\chi_7-1\right) \\
 -2 \mathcal{G}_{7,2}\Big|S_{7,3} -2 \mathcal{G}_{7,4}\Big|S_{7,-2} -2 \mathcal{G}_{7,1}\Big|S_{7,-1} + \frac{1}{4}g_7.
\end{multline*}
\end{corollary}

We do not work out the details for the cusp forms for $a\not\equiv 0 \pmod{7}$, but rather list the resulting identity when restricting to coefficients in certain congruence classes where the coefficients of the cusp forms are all zero.  This suffices to prove Conjecture \ref{conj:7}.
\begin{corollary}\label{cor:7case}
The following identities hold.

\noindent
\begin{enumerate}
\item
For $r\equiv 1,3,6\pmod{7}$, one has
$$
\left(\mathcal{H}\vartheta_{1,7}\right)\Big|U(4)\Big|S_{7,r} = \frac{1}{4}\mathcal{D} + 
\begin{cases}
\frac{1}{12}\mathcal{D}-\mathcal{G}_{7,2}-\mathcal{G}_{7,3}&\text{if }r\equiv 1\pmod 7,\\
0 & \text{if }r\equiv 3,6\pmod{7}.
\end{cases}
$$
\item For $r\equiv 3,5\pmod{7}$, one has 
$$
\left(\mathcal{H}\vartheta_{2,7}\right)\Big|U(4)\Big|S_{7,r} = \frac{1}{4}\mathcal{D}.
$$
\item
For $r\equiv 5,6\pmod{7}$, one has 
$$
\left(\mathcal{H}\vartheta_{3,7}\right)\Big|U(4)\Big|S_{7,r} = \frac{1}{4}\mathcal{D}.
$$
\end{enumerate}
In particular, Conjecture \ref{conj:7} is true.
\end{corollary}

\end{document}